\documentclass[11pt]{article}

\usepackage{
    amsmath,
    amssymb,
    amsthm,
    fullpage,
    thmtools,
    thm-restate,
    tikz
}
\usepackage[colorlinks=true, allcolors=blue]{hyperref}
\usepackage[capitalize,noabbrev]{cleveref}

\usetikzlibrary{decorations.pathreplacing}
\usetikzlibrary{arrows.meta}

\newtheorem{theorem}{Theorem}[section]

\newtheorem{lemma}[theorem]{Lemma}

\theoremstyle{definition}
\newtheorem{definition}[theorem]{Definition}

\newtheorem{problem}[theorem]{Problem}

\newcommand{\ith}[1]{$#1^{\text{th}}$}
\newcommand{\Homclass}[4]{\Hom_{#1\to#2}(#3,#4)}
\newcommand{\homclass}[4]{\hom_{#1\to#2}(#3,#4)}
\newcommand{\trans}{\text{T}}
\newcommand{\mathematica}{\texttt{Mathematica}}

\newcommand{\even}{\rm even}
\newcommand{\odd}{\rm odd}

\DeclareMathOperator{\diag}{diag}
\DeclareMathOperator{\Hom}{Hom}
\DeclareMathOperator{\ind}{ind}

\newcommand{\change}[1]{\widehat{#1}}

\title{Long paths need not minimize $H$-colorings among trees}

\author{
David Galvin\thanks{Galvin is in part supported by a Simons Collaboration Grant for Mathematicians.}\\
University of Notre Dame\\
{\tt dgalvin1@nd.edu}
\and
Emily McMillon\thanks{McMillon is in part supported by NSF DMS-2303380.}\\
Rice University\\
{\tt em72@rice.edu}
\and
JD Nir\\
Oakland University\\
{\tt jdnir@oakland.edu}
\and
Amanda Redlich\\
University of Massachusetts Lowell\\
{\tt amanda\_redlich@uml.edu}
}

\begin{document}

\maketitle

\begin{abstract}
Given a graph $G$ and a target graph $H$, an \textit{$H$-coloring} of $G$ is an adjacency-preserving vertex map from $G$ to $H$. By appropriate choice of $H$, these colorings can express, for instance, the independent sets or proper vertex colorings of $G$. 

Sidorenko proved that for any $H$, the $n$-vertex star admits at least as many $H$-colorings as any other $n$-vertex tree, but the minimization question remains open in general. For many graphs $H$, path graphs are among the trees with the fewest $H$-colorings, but work of Leontovich and subsequently Csikv\'ari and Lin shows that there is a graph $E_7$ on seven vertices and a target graph $H$ for which there are strictly fewer $H$-colorings of $E_7$ than of the path on seven vertices.

We introduce a new strategy for enumerating homomorphisms from path-like trees to highly symmetric target graphs that allows us to make the previous observations completely explicit and extend them to infinitely many $n$ beyond $n=7$. In particular, we exhibit a target graph $H$ with the property that for each sufficiently large $n$, there is a tree $E_n$ on $n$ vertices that admits strictly fewer $H$-colorings than the path on $n$ vertices.
\end{abstract}

\section{Introduction}

In this paper, we consider $H$-colorings of graphs $G$: labelings of the vertices of $G$ by vertices of $H$ such that if vertices are adjacent in $G$, their labels are adjacent in $H$. More formally, for a simple, loopless graph $G$ on vertex set $V(G)$ and a target graph $H$ on vertex set $V(H)$ (possibly with loops, but without multi-edges), an \emph{$H$-coloring} of $G$ is an adjacency-preserving map  $f:V(G) \to V(H)$ (that is, a map satisfying $f(x) \sim_H f(y)$ whenever $x\sim_G y$). Denote by $\Hom(G,H)$ the set of $H$-colorings of $G$ and by $\hom(G,H)$ the number of $H$-colorings of $G$. Note that an $H$-coloring of $G$ is exactly a homomorphism from $G$ to $H$, explaining our choice of notation.

Many important graph notions can be encoded via $H$-colorings, including \emph{proper $q$-colorings} (using $H=K_q$, the complete graph on $q$ vertices) and \emph{independent} (or \emph{stable}) \emph{sets} (using $H=H_{\ind}$, an edge with one looped endvertex). Lov\'asz's monograph~\cite{Lovasz2012} explores connections between $H$-colorings and graph limits, quasi-randomness, and property testing. The language of $H$-coloring is also ideally suited for the mathematical study of hard-constraint spin models from statistical physics (see e.g.~\cite{BrightwellWinkler1999, BrightwellWinkler2002}).

The following extremal enumerative question for $H$-coloring has a long history: given a family ${\mathcal G}$ of graphs and a target graph $H$, which $G \in {\mathcal G}$ maximizes or minimizes $\hom(G,H)$? This question can be traced back to Birkhoff's attacks on the 4-color conjecture (see~\cite{Birkhoff1912, BirkhoffLewis1946}), but recent attention is due more to Wilf's and (independently) Linial's mid-1980's queries~\cite{Linial1986, Wilf1984} as to which $n$-vertex, $m$-edge graph has the most proper $q$-colorings (i.e.~admits the most $K_q$-colorings). For a taste of the wide variety of results and conjectures on the extremal enumerative $H$-coloring question, see the surveys~\cite{Cutler2012, Zhao2017}.

In this paper we consider the extremal enumerative $H$-coloring question for the family ${\mathcal T}_n$, the set of all trees on $n$ vertices. For fixed $H$, adding edges to a graph $G$ weakly decreases the number of $H$-colorings that $G$ admits. The $n$-vertex connected graph that admits the most $H$-colorings, then, must be minimally connected, i.e.~some member of ${\mathcal T}_n$. This family has two natural candidates for extremality, namely the path $P_n$ and the star $S_n=K_{1,n-1}$, and indeed Prodinger and Tichy~\cite{ProdingerTichy1982} showed that these are the extremal trees when counting independent sets: for all $T_n \in {\mathcal T}_n$,
\begin{equation} \label{inq:PT}
\hom(P_n,H_{\ind}) \le \hom(T_n,H_{\ind}) \le \hom(S_n,H_{\ind}).
\end{equation}

The Hoffman-London matrix inequality (see e.g.~\cite{Hoffman1967,London1966,Sidorenko1985}) is equivalent to the statement $\hom(P_n,H) \le \hom(S_n,H)$ for \emph{all} $H$ and $n$. Sidorenko \cite{Sidorenko1994} obtained a significant generalization of the Hoffman-London inequality:
\begin{theorem}[Sidorenko] \label{thm:siderenko}
Fix $H$ and $n \ge 1$.  For any $T_n \in {\mathcal T}_n$,
\[
\hom(T_n,H) \le \hom(S_n,H).
\]
\end{theorem}
In other words, for any target graph $H$, the $n$-star not only admits at least as many $H$-colorings as the path but at least as many as any tree on $n$ vertices. See also~\cite{CsikvariLin2014,LevinPeres2017,LuchtrathMonch2024} for alternate proofs, and~\cite{EngbersGalvin2017} for a proof valid for large $n$.

It is tempting to speculate that for an arbitrary $H$, the $n$-vertex path should admit the fewest $H$-colorings among $n$-vertex trees; the star and the path are, both heuristically and in some quantifiable senses, the most extreme trees, so if the star is a maximizer for any $H$, should not the path always be a minimizer? However, this speculation is incorrect. Csikv\'{a}ri and Lin~\cite{CsikvariLin2014}, following earlier work of Leontovich~\cite{Leontovich1989}, show that there exists a target graph $H$ and a tree $E_7$ on seven vertices such that $\hom(E_7,H) < \hom(P_7,H)$. Their results motivate the following definition.
\begin{definition}
We say a graph $H$ is \emph{Leontovich at $n$} if there is a tree $T_n$ on $n$ vertices such that $\hom(T_n,H) < \hom(P_n,H)$ and we say $H$ is \emph{strongly Leontovich} if $H$ is Leontovich at $n$ for all sufficiently large $n$.
\end{definition}
\noindent So the earlier work cited on this problem shows that there is a target graph that is Leontovich at $n=7$.

The present paper has two goals: we first make the observations of Csikv\'{a}ri and Lin and of Leontovich completely explicit by identifying a specific $H$ that is Leontovich at $n=7$, and then we extend their observations to infinitely many $n$ beyond seven. For each of these goals, we use the same basic construction for $H$, which we now introduce. For integers $x, y, z$, let $T(x,y,z)$ be the rooted tree where the root has $x$ children, each child of the root has $y$ children, and each grandchild of the root has $z$ children (see \cref{fig:Txyz}). Note that $T(x,y,z)$ is an instance of a \emph{spherically symmetric} tree --- a rooted tree in which the number of children a vertex has depends only on its distance from the root.

\begin{figure}[ht!]
    \centering
    \begin{tikzpicture}
        \coordinate (z1) at (0,0) {};
        \coordinate (z2) at (0.25,0);
        \coordinate (z3) at (1,0);
        \coordinate (z4) at (1.25,0);
        \coordinate (z5) at (1.5,0);
        \coordinate (z6) at (2.25,0);
        \coordinate (z7) at (2.5,0);
        \coordinate (z8) at (2.75,0);
        \coordinate (z9) at (3.5,0);

        \coordinate (z10) at (4,0) {};
        \coordinate (z20) at (4.25,0);
        \coordinate (z30) at (5,0);
        \coordinate (z40) at (5.25,0);
        \coordinate (z50) at (5.5,0);
        \coordinate (z60) at (6.25,0);
        \coordinate (z70) at (6.5,0);
        \coordinate (z80) at (6.75,0);
        \coordinate (z90) at (7.5,0);

        \coordinate (z100) at (8.5,0) {};
        \coordinate (z200) at (8.75,0);
        \coordinate (z300) at (9.5,0);
        \coordinate (z400) at (9.75,0);
        \coordinate (z500) at (10,0);
        \coordinate (z600) at (10.75,0);
        \coordinate (z700) at (11,0);
        \coordinate (z800) at (11.25,0);
        \coordinate (z900) at (12,0);

        \coordinate (y1) at (0.5,2);
        \coordinate (y2) at (1.75,2);
        \coordinate (y3) at (3,2);

        \coordinate (y10) at (4.5,2);
        \coordinate (y20) at (5.75,2);
        \coordinate (y30) at (7,2);

        \coordinate (y100) at (9,2);
        \coordinate (y200) at (10.25,2);
        \coordinate (y300) at (11.5,2);

        \coordinate (x1) at (2.25,4);
        \coordinate (x10) at (6.25,4);
        \coordinate (x100) at (10.75,4);

        \coordinate (r) at (8,6);

        \draw[fill=black,thick] (z1)--(y1);
        \draw[fill=black,thick] (z2)--(y1);
        \draw[fill=black,thick] (z3)--(y1);
        \draw[fill=black,thick,dotted] (0.5,0.5)--(0.75,0.5);
        \draw[fill=black,thick] (z4)--(y2);
        \draw[fill=black,thick] (z5)--(y2);
        \draw[fill=black,thick] (z6)--(y2);
        \draw[fill=black,thick,dotted] (1.75,0.5)--(2,0.5);
        \draw[fill=black,thick] (z7)--(y3);
        \draw[fill=black,thick] (z8)--(y3);
        \draw[fill=black,thick] (z9)--(y3);
        \draw[fill=black,thick,dotted] (3,0.5)--(3.25,0.5);

        \draw[fill=black,thick] (z10)--(y10);
        \draw[fill=black,thick] (z20)--(y10);
        \draw[fill=black,thick] (z30)--(y10);
        \draw[fill=black,thick,dotted] (4.5,0.5)--(4.75,0.5);
        \draw[fill=black,thick] (z40)--(y20);
        \draw[fill=black,thick] (z50)--(y20);
        \draw[fill=black,thick] (z60)--(y20);
        \draw[fill=black,thick,dotted] (5.75,0.5)--(6,0.5);
        \draw[fill=black,thick] (z70)--(y30);
        \draw[fill=black,thick] (z80)--(y30);
        \draw[fill=black,thick] (z90)--(y30);
        \draw[fill=black,thick,dotted] (7,0.5)--(7.25,0.5);

        \draw[fill=black,thick] (z100)--(y100);
        \draw[fill=black,thick] (z200)--(y100);
        \draw[fill=black,thick] (z300)--(y100);
        \draw[fill=black,thick,dotted] (9,0.5)--(9.25,0.5);
        \draw[fill=black,thick] (z400)--(y200);
        \draw[fill=black,thick] (z500)--(y200);
        \draw[fill=black,thick] (z600)--(y200);
        \draw[fill=black,thick,dotted] (10.25,0.5)--(10.5,0.5);
        \draw[fill=black,thick] (z700)--(y300);
        \draw[fill=black,thick] (z800)--(y300);
        \draw[fill=black,thick] (z900)--(y300);
        \draw[fill=black,thick,dotted] (11.5,0.5)--(11.75,0.5);

        \draw[fill=black,thick] (y1)--(x1);
        \draw[fill=black,thick] (y2)--(x1);
        \draw[fill=black,thick] (y3)--(x1);
        \draw[fill=black,thick,dotted] (2.25,2.5)--(2.5,2.5);

        \draw[fill=black,thick] (y10)--(x10);
        \draw[fill=black,thick] (y20)--(x10);
        \draw[fill=black,thick] (y30)--(x10);
        \draw[fill=black,thick,dotted] (6.25,2.5)--(6.5,2.5);

        \draw[fill=black,thick] (y100)--(x100);
        \draw[fill=black,thick] (y200)--(x100);
        \draw[fill=black,thick] (y300)--(x100);
        \draw[fill=black,thick,dotted] (10.75,2.5)--(11,2.5);

        \draw[fill=black,thick] (x1)--(r);
        \draw[fill=black,thick] (x10)--(r);
        \draw[fill=black,thick] (x100)--(r);
        \draw[fill=black,thick,dotted] (8.25,4.5)--(8.5,4.5);

        \draw[fill=black] (z1) circle (1.5pt);
        \draw[fill=black] (z2) circle (1.5pt);
        \draw[fill=black] (z3) circle (1.5pt);
        \draw[fill=black] (z4) circle (1.5pt);
        \draw[fill=black] (z5) circle (1.5pt);
        \draw[fill=black] (z6) circle (1.5pt);
        \draw[fill=black] (z7) circle (1.5pt);
        \draw[fill=black] (z8) circle (1.5pt);
        \draw[fill=black] (z9) circle (1.5pt);
        \draw[fill=black] (z10) circle (1.5pt);
        \draw[fill=black] (z20) circle (1.5pt);
        \draw[fill=black] (z30) circle (1.5pt);
        \draw[fill=black] (z40) circle (1.5pt);
        \draw[fill=black] (z50) circle (1.5pt);
        \draw[fill=black] (z60) circle (1.5pt);
        \draw[fill=black] (z70) circle (1.5pt);
        \draw[fill=black] (z80) circle (1.5pt);
        \draw[fill=black] (z90) circle (1.5pt);
        \draw[fill=black] (z100) circle (1.5pt);
        \draw[fill=black] (z200) circle (1.5pt);
        \draw[fill=black] (z300) circle (1.5pt);
        \draw[fill=black] (z400) circle (1.5pt);
        \draw[fill=black] (z500) circle (1.5pt);
        \draw[fill=black] (z600) circle (1.5pt);
        \draw[fill=black] (z700) circle (1.5pt);
        \draw[fill=black] (z800) circle (1.5pt);
        \draw[fill=black] (z900) circle (1.5pt);

        \draw[fill=black] (y1) circle (1.5pt);
        \draw[fill=black] (y2) circle (1.5pt);
        \draw[fill=black] (y3) circle (1.5pt);
        \draw[fill=black] (y10) circle (1.5pt);
        \draw[fill=black] (y20) circle (1.5pt);
        \draw[fill=black] (y30) circle (1.5pt);
        \draw[fill=black] (y100) circle (1.5pt);
        \draw[fill=black] (y200) circle (1.5pt);
        \draw[fill=black] (y300) circle (1.5pt);

        \draw[fill=black] (x1) circle (1.5pt);
        \draw[fill=black] (x10) circle (1.5pt);
        \draw[fill=black] (x100) circle (1.5pt);

        \draw[fill=black] (r) circle (1.5pt);

        \node at (13.5,6) {\begin{tabular}{c} {\small 1 vertex} \\ {\small down degree $x$} \end{tabular}};

        \node at (13.5,4) {\begin{tabular}{c} {\small $x$ vertices} \\ {\small down degree $y$} \end{tabular}};

        \node at (13.5,2) {\begin{tabular}{c} {\small $xy$ vertices} \\ {\small down degree $z$} \end{tabular}};

        \node at (13.5,0) {\begin{tabular}{c} {\small $xyz$ vertices}\end{tabular}};
    \end{tikzpicture}
    \caption{A visual representation of the graph $T(x,y,z)$.} \label{fig:Txyz}
\end{figure}
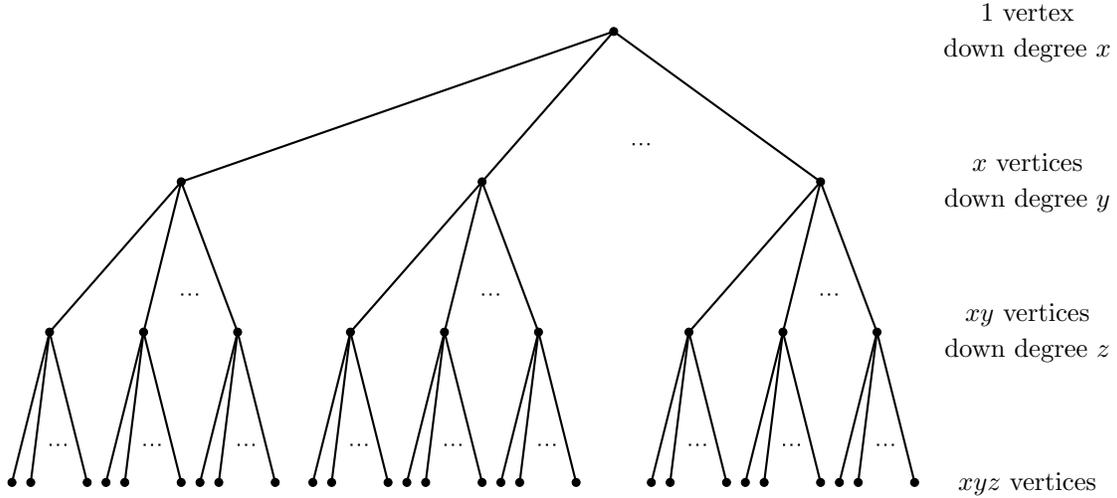

Let $E_7$ be the tree obtained from $P_6$ by adding a pendant edge from one of the vertices that is third from the end of the path (see \cref{fig:e7-p7}). Leontovich~\cite{Leontovich1989} observed that there is some triple $(x,y,z)$ that yields the inequality
\begin{equation} \label{inq-Leontovich}
\hom(E_7,T(x,y,z)) < \hom(P_7,T(x,y,z)).
\end{equation}
Leontovich's observation was fleshed out by Csikv\'ari and Lin~\cite{CsikvariLin2014}, who observe that inequality \eqref{inq-Leontovich} can be achieved (meaning $T(x,y,z)$ is Leontovich at $n=7$) by taking sufficiently large $x, y, z$ satisfying $y \ll x \ll z \ll xy$, but they do not provide explicit values for $x$, $y$, and $z$. We make the observations of Leontovich and of Csikv\'ari and Lin explicit in the next result.

\begin{restatable}{theorem}{explicitexample} \label{thm:first_explicit_example}
The graph $T(18,3,32)$ is Leontovich at $n=7$. Specifically, we have 
\[\hom(E_7,T(18,3,32)) < \hom(P_7,T(18,3,32)).\]
\end{restatable}

Note that $T(18,3,32)$ has $1801$ vertices---far more than $E_7$ and $P_7$. It is tempting to speculate that for each graph $H$ and all sufficiently large $n$, we have $\hom(P_n,H) \le \hom(T_n,H)$ for all $T_n \in {\mathcal T}_n$.  Csikv\'ari and Lin~\cite[Theorem 3.3 and following discussion]{CsikvariLin2014} noted that this speculation is true asymptotically, at least for connected $H$, in the following sense: for each connected $H$ there is a constant $\lambda \ge 0$ (the largest eigenvalue of the adjacency matrix of $H$) such that if $(T_n)_{n \in {\mathbb N}}$ is any sequence of trees with $T_n \in {\mathcal T}_n$ then
\[\lim_{n \rightarrow \infty} \hom(P_n,H)^{1/n} = \lambda \leq \liminf_{n \rightarrow \infty} \hom(T_n,H)^{1/n}. 
\]
So although the path may not always be the minimizer, no infinite family can ``substantially'' beat it.

We show, however, that this speculation is incorrect in an absolute sense, first with a target graph $H$ that is Leontovich at infinitely many $n$ and then with a strongly Leontovich $H$.
In \cref{thm:path_not_eventually_minimizer} below, $E_n$ is the $n$-vertex tree obtained from $P_{n-1}$ by adding a pendant edge to a third-from-last vertex of $P_{n-1}$ (see \cref{fig:gn}) and $\widehat{T}(x,y,z)$ is obtained from $T(x,y,z)$ by adding a loop at the root.

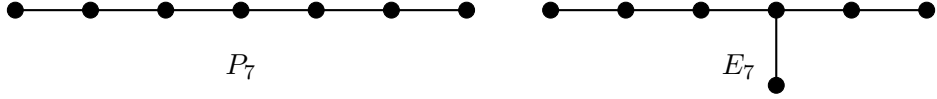
\begin{figure}[ht!]
    \centering
    \begin{tikzpicture}
        \coordinate (v1) at (1,0);
        \coordinate (v2) at (2,0);
        \coordinate (v3) at (3,0);
        \coordinate (v4) at (4,0);
        \coordinate (v5) at (5,0);
        \coordinate (v6) at (6,0);
        \coordinate (v7) at (7,0);
        \draw[fill=black] (v1) circle (3pt);
        \draw[fill=black] (v2) circle (3pt);
        \draw[fill=black] (v3) circle (3pt);
        \draw[fill=black] (v4) circle (3pt);
        \draw[fill=black] (v5) circle (3pt);
        \draw[fill=black] (v6) circle (3pt);
        \draw[fill=black] (v7) circle (3pt);
        \draw[fill=black,thick] (v1)--(v2)--(v3)--(v4)--(v5)--(v6)--(v7);
        \node at (4,-0.75) {$P_7$};
        \coordinate (dummy) at (6,-1);
        \draw[fill=white,white] (dummy) circle (3pt);
    \end{tikzpicture}
    \hspace{0.25in}
    \begin{tikzpicture}
        \coordinate (v1) at (1,0);
        \coordinate (v2) at (2,0);
        \coordinate (v3) at (3,0);
        \coordinate (v4) at (4,0);
        \coordinate (v5) at (5,0);
        \coordinate (v6) at (6,0);
        \coordinate (v7) at (4,-1);
        \draw[fill=black] (v1) circle (3pt);
        \draw[fill=black] (v2) circle (3pt);
        \draw[fill=black] (v3) circle (3pt);
        \draw[fill=black] (v4) circle (3pt);
        \draw[fill=black] (v5) circle (3pt);
        \draw[fill=black] (v6) circle (3pt);
        \draw[fill=black] (v7) circle (3pt);
        \draw[fill=black,thick] (v1)--(v2)--(v3)--(v4)--(v5)--(v6);
        \draw[fill=black,thick] (v4)--(v7);
        \node at (3.5,-0.75) {$E_7$};
    \end{tikzpicture}
    \caption{$P_7$ and $E_7$.} \label{fig:e7-p7}
\end{figure}

\begin{restatable}{theorem}{pathnoteventuallyminimizer}
\label{thm:path_not_eventually_minimizer}
For all sufficiently large odd $n$, we have
\[
\hom(E_n,T(7,1,9)) < \hom(P_n,T(7,1,9)).
\]
Further, for all sufficiently large $n$, odd or even, we have
\[
\hom(E_n,\widehat{T}(400,3,800)) < \hom(P_n,\widehat{T}(400,3,800)).
\]
In other words, $T(7,1,9)$ is Leontovich for all sufficiently large odd $n$ and $\widehat{T}(400,3,800)$ is strongly Leontovich.
\end{restatable}

\begin{figure}[ht]
    \centering
    \begin{tikzpicture}
        \coordinate (v1) at (1,0);
        \coordinate (v2) at (2,0);
        \coordinate (v3) at (3,0);
        \coordinate (v4) at (4,0);
        \coordinate (v5) at (5,0);
        \coordinate (v6) at (6,0);
        \coordinate (v7) at (7,0);
        \coordinate (v8) at (8,0);
        \draw[fill=black] (v1) circle (3pt);
        \draw[fill=black] (v2) circle (3pt);
        \draw[fill=black] (v3) circle (3pt);
        \draw[fill=black] (v4) circle (3pt);
        \draw[fill=black] (v5) circle (3pt);
        \draw[fill=black] (v6) circle (3pt);
        \draw[fill=black] (v7) circle (3pt);
        \draw[fill=black] (v8) circle (3pt);
        \draw[fill=black,thick] (v1)--(v2);
        \draw[fill=black,thick] (v3)--(v4)--(v5)--(v6)--(v7)--(v8);
        \draw[fill=black,thick,dashed] (v2) -- (v3);
        \draw [decorate,decoration={brace,amplitude=5pt,raise=10pt}] (1,0) -- (4,0) node[midway,yshift=23pt]{$P_{n}$};
        \node at (4.5,-0.75) {$P_{n+4}$};
        \coordinate (dummy) at (6,-1);
        \draw[fill=white,white] (dummy) circle (3pt);
    \end{tikzpicture}
    \hspace{0.25in}
    \begin{tikzpicture}
        \coordinate (v1) at (1,0);
        \coordinate (v2) at (2,0);
        \coordinate (v3) at (3,0);
        \coordinate (v4) at (4,0);
        \coordinate (v5) at (5,0);
        \coordinate (v6) at (6,0);
        \coordinate (v7) at (5,-1);
        \coordinate (v8) at (7,0);
        \draw[fill=black] (v1) circle (3pt);
        \draw[fill=black] (v2) circle (3pt);
        \draw[fill=black] (v3) circle (3pt);
        \draw[fill=black] (v4) circle (3pt);
        \draw[fill=black] (v5) circle (3pt);
        \draw[fill=black] (v6) circle (3pt);
        \draw[fill=black] (v7) circle (3pt);
        \draw[fill=black] (v8) circle (3pt);
        \draw[fill=black,thick] (v1)--(v2);
        \draw[fill=black,thick] (v3)--(v4)--(v5)--(v6)--(v8);
        \draw[fill=black,thick] (v5)--(v7);
        \draw[fill=black,dashed,thick] (v2)--(v3);
        \draw [decorate,decoration={brace,amplitude=5pt,raise=10pt}] (1,0) -- (4,0) node[midway,yshift=23pt]{$P_{n}$};
        \node at (4,-0.75) {$E_{n+4}$};
    \end{tikzpicture}
    \caption{The graphs $P_{n+4}$ (later referred to as $G_n$) and $E_{n+4}$ (later referred to as $G'_n$).} \label{fig:gn}
\end{figure}
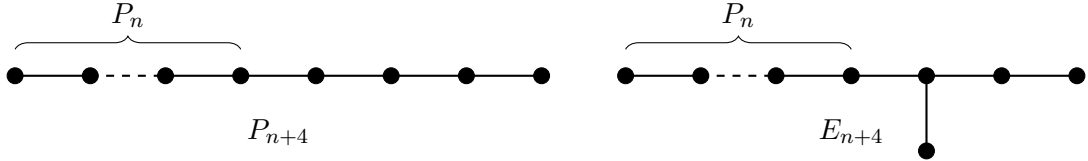

We pause briefly to address what happens for $T(7,1,9)$ for even $n$. It turns out that for all sufficiently large even $n$, we have
\[
\hom(P_n, T(7,1,9)) < \hom(E_n, T(7,1,9)).
\]
Despite extensive computational search, we have not found values $(x,y,z)$ for which the inequality $\hom(E_n, T(x,y,z)) < \hom(P_n, T(x,y,z))$ holds for all sufficiently large $n$. This deficiency is remedied by our second choice of $H$. Note that if one desires $H$ to be a simple graph, then one can replace $\widehat{T}(x,y,z)$ with two copies of $T(x,y,z)$ with an edge joining their roots. The analysis that we will use in \cref{thm:path_not_eventually_minimizer} will go through almost identically, the only difference being that in this case the homomorphism counts are scaled by a factor of $2$.

\medskip

To effectively handle enumeration of homomorphisms to the graphs $T(x,y,z)$, we make use of the \emph{orbit partition} of a graph. This is the partition of the vertex set obtained by declaring a pair of vertices $u, v$ to be equivalent if there is an automorphism that maps $u$ to $v$. While the orbit partition has played an important role in the graph isomorphism problem (see e.g. \cite{CaiFurerImmerman1992, DawarKhan2029}), it has, to present, not been used in the study of graph homomorphism enumeration. The orbit partition and its associated \emph{automorphic similarity matrix} allow us to exploit the symmetry of $T(x,y,z)$. We describe this and our other technical tools in \cref{sec:main_theorem_proofs}, and use these tools to establish two general results, \cref{thm:extension_framework} and \cref{cor:coefficients}, about the enumeration of $H$-colorings. In \cref{sec:extension_applications} we use these results to establish \cref{thm:first_explicit_example} and \cref{thm:path_not_eventually_minimizer}. We close with a collection of open problems in \cref{sec:open_problems}.    

In a companion paper~\cite{GalvinMarmorinoMcMillonNirRedlich2025}, we consider a complementary question: for which target graphs $H$ do we observe that $\hom(P_n,H) \le \hom(T_n,H)$ for all $n$ and all $T_n \in \mathcal{T}_n$? Using the notion of automorphic similarity, we establish a general criterion to approach this problem.      

\section{Automorphic similarity and path-like families} \label{sec:main_theorem_proofs}

In this section we present a strategy for efficiently calculating the number of $H$-colorings of certain families of graphs by using the concept of automorphic similarity, which we now define.

Let $H$ be a target graph, possibly with loops. For vertices $u$ and $v$ in $H$, we say $u$ and $v$ are \emph{automorphically similar} if there is an automorphism of $H$ that sends $u$ to $v$. Automorphic similarity defines an equivalence relation on $V(H)$, and we call the equivalence classes of that relation the \emph{automorphic similarity classes of $H$}. We denote these equivalence classes as $H^1, \ldots, H^k$, where $k$ is the number of automorphism orbits of $V(H)$. The partition $V(H)=\bigcup_i H^i$, sometimes referred to as the \emph{orbit partition}, is an instance of an \emph{equitable partition} of the graph $H$ (see e.g. \cite[Section 9.3]{GodsilRoyle2001}). This is a partition with the property that for any two classes $A, B$ (not necessarily distinct) there is a constant $c(A,B)$ with the property that every vertex in $A$ has $c(A,B)$ neighbors in $B$. For completeness, we present the next lemma, which verifies that $\bigcup_i H^i$ is an equitable partition.    
\begin{lemma}\label{lem:auto_sim_nbrs}
Let $H$ be a graph, possibly with loops, and let $H^1, \ldots, H^k$ be the automorphic similarity classes of $H$. For any $1 \le i \le k$, any $u,v \in H^i$, and any $1 \le j \le k$,
\[ |N(u) \cap H^j| = |N(v) \cap H^j|.\]
\end{lemma}

\begin{proof}
Let $\varphi$ be an automorphism of $V(H)$ that sends $u$ to $v$. Then $\varphi$ maps each $w \in N(u) \cap H^j$ to some $w'$. As $w \sim u$, we see $w' \sim v$, and as $w \in H^j$ and $\varphi(w) = w'$ we see $w$ and $w'$ are automorphically similar, so $w' \in H^j$ as well. We have shown $|N(u) \cap H^j| \le |N(v) \cap H^j|$. Repeating the analysis with $\varphi^{-1}$, we see $|N(u) \cap H^j| \ge |N(v) \cap H^j|$ and so equality holds.
\end{proof}

\cref{lem:auto_sim_nbrs} permits us to define the numbers $m_{i,j}$ to be the number of neighbors any $v \in H^i$ has in $H^j$. The \emph{automorphic similarity matrix} of $H$ (relative to the particular order put on the automorphic similarity classes) is the $k \times k$ matrix $M = M(H)$ whose $(i,j)$-entry is $m_{i,j}$. This matrix is also referred to as the \emph{quotient matrix} of the adjacency matrix of $H$ relative to the equitable partition $\bigcup_i H^i$. 

When $H$ is highly symmetric, the automorphic similarity matrix of $H$ is much smaller than the adjacency matrix and thus easier to analyze. When $T$ is a tree, the automorphic similarity matrix, together with the sizes of the automorphic similarity classes, contains enough information to compute $\hom(T, H)$. This general fact is proved in \cite{GalvinMarmorinoMcMillonNirRedlich2025}, but is not needed in this paper. Here we work only with families of path-like trees (to be defined shortly), and \cref{thm:extension_framework} below gives us all that we need in terms of counting $H$-colorings of these trees.  

For a graph $G$ with designated vertex $v$, let $\Homclass{v}{i}{G}{H}$ denote the set of $H$-colorings of $G$ in which $v$ is mapped to some (arbitrary but specific) representative of $H^i$. We use $\homclass{v}{i}{G}{H}$ for the number of such $H$-colorings, and we define $\mathbf{h}(G,v)$ to be the column vector whose \ith{i} entry is $\homclass{v}{i}{G}{H}$.

Although $\Homclass{v}{i}{G}{H}$ clearly depends on the choice of representative, we show in the following lemma that $\homclass{v}{i}{G}{H}$ does not. This allows us to extract crucial information from the automorphic similarity matrix.

\begin{lemma} \label{lem:using_iso}
Let $H$ be a graph, possibly with loops, and let $G$ be an arbitrary graph with distinguished vertex $v$. Suppose $w$ and $w'$ are automorphically similar vertices in $H$. Let ${\mathcal G}$ be the set of $H$-colorings of $G$ in which $v$ is sent to $w$, and let ${\mathcal G}'$ be the set of $H$-colorings of $G$ in which $v$ is sent to $w'$. Then $|{\mathcal G}|=|{\mathcal G}'|$.
\end{lemma}

\begin{proof}
Let $\varphi$ be an automorphism of $H$ that sends $w$ to $w'$, which exists as $w$ and $w'$ are automorphically similar. Then the function mapping $f$ to $f \circ \varphi$ is a bijection mapping ${\mathcal G}$ to ${\mathcal G}'$. 
\end{proof}

We now define the families of graphs that will be our main concern, and then state and prove our main technical theorem. Let $G_0$ be a graph (not necessarily a tree, although in each of our applications $G_0$ will be a tree) and fix a vertex $v_0 \in G_0$. For $n \ge 0$, define $G_n$ to be the graph formed by appending a path with $n$ new vertices to $v_0$. We refer to the family of graphs $\{G_n\}_{n=0}^\infty$ as a \emph{path-like family} of graphs, and we call the initial graph $G_0$ the \emph{seed graph} of the family. Note that by taking $G_0$ to be a single vertex, we see the set of all paths form a path-like family.

\begin{restatable}{theorem}{extensions} \label{thm:extension_framework}
Let $H$ be a target graph, possibly with loops, and let $H^1, \ldots, H^k$ be the automorphic similarity classes of $H$, with $M$ the associated automorphic similarity matrix. Denote by $\mathbf{a}(H)$ the row vector whose \ith{i} entry is $|H^i|$. Let $G_0$ be a graph, fix a vertex $v_0 \in G_0$, and let $\{G_n\}_{n=0}^\infty$ be the path-like family with seed graph $G_0$ and fixed vertex $v_0$.
Then
\[ \hom(G_n,H) = \mathbf{a}(H) M^n \mathbf{h}(G_0,v_0).\]
\end{restatable}

\begin{proof}
For $n \ge 1$, let $v_n$ be the unique vertex in $V(G_n) \setminus V(G_{n-1})$, the leaf at the end of the path extending from $v_0$. By \cref{lem:using_iso}, $\homclass{v_n}{i}{G_n}{H}$ is independent of the choice of representative of $H^i$. By partitioning the $H$-colorings of $G_n$ by the $H^i$ to which we send $v_n$, we can write
\begin{equation} \label{eq:first_recursion_step}
\hom(G_n,H) = \sum_{i=1}^k |H^i| \homclass{v_n}{i}{G_n}{H} = \mathbf{a}(H) \mathbf{h}(G_n,v_n).
\end{equation}
Note that equation \eqref{eq:first_recursion_step} is valid for all $n \ge 0$, and at $n=0$ it gives us our claimed expression for $\hom(G_0,H)$.

For $n \ge 1$, we can further partition the $H$-colorings of $G_n$ in which $v_n$ is sent to $H^i$ by the automorphic similarity class to which $v_{n-1}$ is sent, where $v_{n-1}$ is the unique neighbor of $v_n$ in $G_n$. More specifically, we can map $v_{n-1}$ to one of $m_{i,j}$ vertices in $H^j$ and then color $G_{n-1}$ in $\homclass{v_{n-1}}{j}{G_{n-1}}{H}$ ways. Summing over all choices of $j$, we see
\[ \homclass{v_n}{i}{G_n}{H} = \sum_{j=1}^k m_{i,j} \homclass{v_{n-1}}{j}{G_{n-1}}{H} \]
or $\mathbf{h}(G_n,v_n) = M \mathbf{h}(G_{n-1},v_{n-1})$
Applying this idea recursively gives $\mathbf{h}(G_n,v_n) = M^n \mathbf{h}(G_0,v_0)$. Plugging this equation into equation \eqref{eq:first_recursion_step} gives
\begin{equation} \label{eq:second_recursion_step}
\hom(G_n,H) = \mathbf{a}(H) M^n \mathbf{h}(G_0,v_0),
\end{equation}
as claimed.
\end{proof}

One consequence of \cref{thm:extension_framework} is that if $G_n$ is a member of a path-like family then the number of $H$-colorings of $G_n$ can be written as a linear combination of \ith{n} powers of the eigenvalues of $M$, the automorphic similarity matrix of $H$. The details are given in \cref{cor:coefficients}, which we now state and prove.

\begin{restatable}{corollary}{coefficients} \label{cor:coefficients}
Let $H$ be a target graph, possibly with loops. Let $G_0$ be a graph, fix a vertex $v_0 \in G_0$, and let $\{G_n\}_{n=0}^\infty$ be the path-like family with seed graph $G_0$ and fixed vertex $v_0$. There are real numbers $\lambda_1, \ldots, \lambda_k$ dependent only on $H$, and real numbers $c_1, \ldots, c_k$ dependent on $G_0$, $v_0$, and $H$ such that
\[ \hom(G_n,H) = \sum_{i=1}^k c_i\lambda_i^n.\]
\end{restatable}

\begin{proof}
The automorphic similarity matrix $M$ is not necessarily symmetric, but it is similar to a symmetric matrix. Indeed, let $B={\rm diag}(\sqrt{h_1}, \sqrt{h_2}, \ldots, \sqrt{h_k})$ where $h_i=|H^i|$. The $(i,j)$-entry of $BMB^{-1}$ is $\sqrt{h_i}\, m_{i,j}/\sqrt{h_j}$, so symmetry is implied by the equality 
\[
\frac{\sqrt{h_i}\, m_{i,j}}{\sqrt{h_j}} = \frac{\sqrt{h_j} \, m_{j,i}}{\sqrt{h_i}}
\] 
or equivalently $h_im_{i,j}=h_jm_{j,i}$. This equality holds because both sides count the number of edges between $H^i$ and $H^j$. Since $BMB^{-1}$ is real and symmetric, it is diagonalizable, and so $M$ is diagonalizable, say via $M=ADA^{-1}$ with $D=\diag(\lambda_1, \ldots, \lambda_k)$, all $\lambda_i$ real. Applying this identity to the count given in \cref{thm:extension_framework} gives
\[\hom(G_n,H) = \mathbf{a}(H) A D^n A^{-1} \mathbf{h}(G_0,v_0).\]
The right-hand side can be written as a linear combination of the $\lambda_i^n$. The coefficient $c_i$ of $\lambda_i^n$ is determined by $\mathbf{a}(H)$ and $A$, which depends on $H$, and by $\mathbf{h}(G_0,v_0)$, which depends on $G_0$, $v_0$, and $H$, so the result follows.
\end{proof}
 
By diagonalizing $M$, we have expressed the number of $H$-colorings of path-like families as a sum of a linear (in $k$) number of terms, each of which is exponential in $n$. When $n$ is sufficiently large, that sum is dominated by the largest (in magnitude) eigenvalues whose coefficients are non-zero. This allows for straightforward comparisons of the asymptotic behavior of $H$-colorings of path-like different families.

\section{Extending the observations of Leontovich, Csikv\'ari, and Lin} \label{sec:extension_applications}

Recall from the introduction that there is no analog of \cref{thm:siderenko} for the left-hand side of inequality \eqref{inq:PT}. That is, there are target graphs $H$ and values $n$ for which the $n$-vertex path does \emph{not} have the minimum number of $H$-colorings among $n$ vertex trees: the graphs that are Leontovich at $n$. Leontovich~\cite{Leontovich1989} and later Csikv\'{a}ri and Lin~\cite{CsikvariLin2014} observed that there exist target graphs that are Leontovich at $n=7$. In this section we first make this observation concrete, then significantly extend it.

Throughout this section, let $G_0 = P_4$ be the seed graph of the following two path-like families of graphs: $\{G_n\}_{n=0}^\infty$, which is built by appending a path of length $n$ to a leaf of $G_0$, and $\{G_n'\}_{n=0}^\infty$, which is built by appending a path of length $n$ to a non-leaf of $G_0$. (See \cref{fig:gn}.) Note that $G_3 = P_7$ and $G_3' = E_7$, and more generally, $G_i = P_{4+i}$ and $G_i' = E_{4+i}$, and so these families are natural extensions of the graphs given in \cref{fig:e7-p7}.

We begin with \cref{thm:first_explicit_example}, which we restate for convenience.

\explicitexample*

\begin{proof}
We use \cref{thm:extension_framework} to count $\hom(E_7,T(18,3,32))$ and $\hom(P_7,T(18,3,32))$.
Note that throughout this proof we use $H$ and $T(x,y,z)$ interchangeably; in particular, $H$ does not represent a general target graph. It is straightforward to confirm that $H$ has four automorphic similarity classes corresponding to the root and three generations of children in $H$. Let $H^i$ be the collection of vertices at distance $i-1$ from the root. This gives
\begin{equation} \label{eq:a-for-T(x,y,z)}
\mathbf{a}(T(x,y,z)) = \bigl[|H^i|\bigr]_{i=1}^4 = \begin{bmatrix} 1 & x & xy & xyz \end{bmatrix}
\end{equation}
with the corresponding automorphic similarity matrix
\begin{equation} \label{eq:M-for-T(x,y,z)}
M =
\begin{bmatrix}
    0 & x & 0 & 0 \\
    1 & 0 & y & 0 \\
    0 & 1 & 0 & z \\
    0 & 0 & 1 & 0
\end{bmatrix}.
\end{equation}

Recall our definitions of $\homclass{v}{i}{T}{H}$, the number of $H$-colorings of $T$ where a designated vertex $v$ is sent to an arbitrary but specific element of $H^i$, and $\mathbf{h}(T,v)$, the $k$-dimensional column vector whose \ith{i} entry is $\homclass{v}{i}{T}{H}$. For a given $T$ and $v$, we can calculate $\mathbf{h}(T,v)$ recursively: if $T = K_1$ is a single vertex, then $\mathbf{h}(T,v) = \begin{bmatrix} 1 & 1 & 1 & 1\end{bmatrix}^\trans$. Otherwise, let $u_1, \ldots, u_d$ be the neighbors of $v$ and define $S_i$ to be the component of $T-v$ containing $u_i$. Then
\[
\mathbf{h}(T,v) =
\begin{bmatrix}
\displaystyle \prod_{i=1}^d x\homclass{u_i}{2}{S_i}{H}\\
\displaystyle\prod_{i=1}^d (\homclass{u_i}{1}{S_i}{H} + y\homclass{u_i}{3}{S_i}{H})\\
\displaystyle\prod_{i=1}^d (\homclass{u_i}{2}{S_i}{H} + z\homclass{u_i}{4}{S_i}{H})\\
\displaystyle\prod_{i=1}^d \homclass{u_i}{3}{S_i}{H}
\end{bmatrix}.
\]
For example, the second entry is calculated by noting that if an $H$-coloring of $T$ maps $v$ to $H^2$, each neighbor of $v$ can be mapped either to the root (in one way) or to a child of the target of $v$ (in $y$ ways), with the choices made independently. The other entries are computed similarly.

Applying this recursive method to $G_0 = P_4$ where $u_0$ is a leaf and $v_0$ is a non-leaf gives
\[ 
\mathbf{h}(G_0,u_0) = 
\begin{bmatrix}
x(x + y (1 + z))\\
x(1 + y) + y(1 + y + z)\\
x + y(1 + z)+ z(1 + z)\\
1 + y + z
\end{bmatrix}
\quad
\text{and}
\quad
\mathbf{h}(G_0,v_0) = 
\begin{bmatrix}
x(1 + y)x\\
(x + y(1 + z))(1 + y)\\
(1+ y + z)(1 + z)\\
1 + z
\end{bmatrix}.\]

We now apply \cref{thm:extension_framework}, setting $x=18, y=3$, and $z=32$, to calculate
\[ \hom(P_7,T(18,3,32)) = \mathbf{a}(T(18,3,32)) M^3 \mathbf{h}(G_0,u_0) = 81558090\]
and
\[ \hom(E_7,T(18,3,32)) = \mathbf{a}(T(18,3,32)) M^3 \mathbf{h}(G_0,v_0) = 81548856,\]
from which it follows that $\hom(E_7,T(18,3,32)) < \hom(P_7,T(18,3,32))$.
\end{proof}

We now turn to our second main result, \cref{thm:path_not_eventually_minimizer}, which we also restate for convenience. 

\pathnoteventuallyminimizer*

\begin{proof}
We use \cref{cor:coefficients} to compare the behavior of $\hom(E_n,T(7,1,9))$ and $\hom(P_n,T(7,1,9))$ when $n$ is large and odd, and also to compare $\hom(E_n,\widehat{T}(400,3,800))$ and $\hom(P_n,\widehat{T}(400,3,800))$ for large $n$ of arbitrary parity. 

We begin with $T(7,1,9)$. As in the proof of \cref{thm:first_explicit_example}, we use $H$ to denote $T(x,y,z)$. We have already calculated $\mathbf{a}(H)$ and the automorphic similarity matrix $M$, in equations \eqref{eq:a-for-T(x,y,z)} and \eqref{eq:M-for-T(x,y,z)}. The characteristic polynomial of $M$ is
\[
\det(M-tI) = t^4 - (x+y+z)t^2 +xz.
\]
For $x, y, z > 0$ this polynomial has distinct roots, specifically 
\[
\lambda = \sqrt{\frac{x+y+z + \sqrt{(x+y+z)^2 - 4xz}}{2}} \qquad \text{and} \qquad \mu = \sqrt{\frac{x+y+z - \sqrt{(x+y+z)^2 - 4xz}}{2}},
\]
as well as $-\lambda$ and $-\mu$. By \cref{cor:coefficients}, there are constants $c_1, c_2, d_1, d_2$ such that
\begin{eqnarray} \label{eq:path_expanded}
\hom(P_{n+4},T(x,y,z)) &=& c_1 \lambda^n + d_1 \mu^n + c_2(-\lambda)^n + d_2 (-\mu)^n \notag\\
    &=& (c_1+(-1)^nc_2)\lambda^n + (d_1+(-1)^n d_2)\mu^n
\end{eqnarray}
and constants $c'_1, c'_2, d'_1, d'_2$ such that
\begin{eqnarray} \label{eq:en_expanded} 
\hom(E_{n+4},T(x,y,z)) &=& c'_1 \lambda^n + d'_1 \mu^n + c'_2(-\lambda)^n + d'_2 (-\mu)^n
 \notag\\
    &=& (c'_1+(-1)^nc'_2)\lambda^n + (d'_1+(-1)^n d'_2)\mu^n.
\end{eqnarray}

Define $c_{\odd} = c_1 - c_2$ and $d_{\odd}, c'_{\odd}$, and $d'_{\odd}$ analogously so that for odd $n$,
\begin{equation} \label{eq:compact_exponents}
    \hom(P_{n+4},T(x,y,z)) = c_{\odd} \lambda^n + d_{\odd} \mu^n \qquad \text{and} \qquad \hom(E_{n+4},T(x,y,z)) = c'_{\odd} \lambda^n + d'_{\odd} \mu^n.
\end{equation}
Because $y > 0$,
\[ (x+y+z)^2 -4xz > (x+z)^2-4xz = (x-z)^2 \ge 0, \]
so $\lambda > \mu$. Thus if $c_{\odd} > c'_{\odd}$, then $\hom(P_{n+4},T(x,y,z)) > \hom(E_{n+4},T(x,y,z))$ for all odd $n$ that are sufficiently large.

We could, in principle, explicitly compute the similarity matrix $A=A(x,y,z)$ that diagonalizes $M$ and so explicitly compute each of the constants mentioned above as functions of $x, y$, and $z$. The expressions involved, while algebraic, turn out to be challenging to work with. For example, the $(1,1)$-entry of the similarity matrix computed using \mathematica~is
\[
\frac{\sqrt{x + y + z - \sqrt{x^2 + 2x(y - z) + (y + z)^2}}
\;\left(-x + y + z + \sqrt{x^2 + 2x(y - z) + (y + z)^2}\right)}
{2\sqrt{2}}
\]
and the $(1,1)$-entry of its inverse is
\[
\frac{1}{\sqrt{2}\,\sqrt{x^2 + 2x(y - z) + (y + z)^2}\,
\sqrt{x + y + z - \sqrt{x^2 + 2x(y - z) + (y + z)^2}}}.
\] 

We implement an alternate approach that avoids the challenge of working with algebraic but irrational expressions. Applying equation \eqref{eq:path_expanded} at $n=1$ and $n=3$, we obtain the equalities
\[
\hom(P_5,T(x,y,z)) = c_{\odd}\lambda + d_{\odd}\mu \quad\text{and}\quad\hom(P_7,T(x,y,z)) = c_{\odd}\lambda^3 + d_{\odd}\mu^3
\]
or
\[
\begin{bmatrix}
\hom(P_5,T(x,y,z)) \\
\hom(P_7,T(x,y,z))
\end{bmatrix}
= 
\begin{bmatrix}
\lambda & \mu \\
\lambda^3 & \mu^3
\end{bmatrix}
\begin{bmatrix}
c_{\odd} \\
d_{\odd}
\end{bmatrix}.
\]
By inverting the $2 \times 2$ matrix, noting $\lambda > \mu > 0$, we get 
\[
c_{\odd} = \frac{\hom(P_7,T(x,y,z))-\mu^2\hom(P_5,T(x,y,z))}{\lambda(\lambda^2-\mu^2)}
\]
and, similarly, using equation \eqref{eq:en_expanded} at $n=1$ and $n=3$,
\[
c'_{\odd} = \frac{\hom(E_7,T(x,y,z))-\mu^2\hom(E_5,T(x,y,z))}{\lambda(\lambda^2-\mu^2)}.
\]
We can compute the terms $\hom(P_{n+4},T(x,y,z))$ and $\hom(E_{n+4},T(x,y,z))$ at $n=1,3$ using \cref{thm:extension_framework}, which tells us
that
\[
\hom(P_{n+4},T(x,y,z)) = 
\begin{bmatrix} 1&x&xy&xyz \end{bmatrix}
\begin{bmatrix} 0&x&0&0\\1&0&y&0\\0&1&0&z\\0&0&1&0\end{bmatrix}^n 
\begin{bmatrix}
x(x + y (1 + z))\\
x(1 + y) + y(1 + y + z)\\
x + y(1 + z)+ z(1 + z)\\
1 + y + z
\end{bmatrix}
\]
and
\[
\hom(E_{n+4},T(x,y,z)) = 
\begin{bmatrix} 1&x&xy&xyz \end{bmatrix}
\begin{bmatrix} 0&x&0&0\\1&0&y&0\\0&1&0&z\\0&0&1&0\end{bmatrix}^n 
\begin{bmatrix}
x(1 + y)x\\
(x + y(1 + z))(1 + y)\\
(1+ y + z)(1 + z)\\
1 + z
\end{bmatrix}.
\]
Recalling that $\pm \lambda$ and $\pm \mu$ are the roots of the equation $f(t):=t^4-(x+y+z)t^2+xz=0$ (with $\lambda > \mu > 0$), we can rigorously bound 
\[
\mu_\ell < \mu < \mu_u \quad\text{and}\quad \lambda_\ell < \lambda <\lambda_u
\]
by finding rationals $0 < \mu_\ell < \mu_u < \lambda_\ell < \lambda_u$ satisfying $f(\mu_\ell) > 0 > f(\mu_u)$ and $f(\lambda_\ell) < 0 < f(\mu_u)$. We can then get rational bounds on $c_{\odd}$ and $c'_{\odd}$:
\begin{equation} \label{inq-coddlb}
c_{\odd} > \frac{\hom(P_7,T(x,y,z))-\mu_u^2\hom(P_5,T(x,y,z))}{\lambda_u(\lambda_u^2-\mu_\ell^2)}
\end{equation}
and
\begin{equation} \label{inq-c'oddub}
c'_{\odd} < \frac{\hom(E_7,T(x,y,z))-\mu_\ell^2\hom(E_5,T(x,y,z))}{\lambda_\ell(\lambda_\ell^2-\mu_u^2)}.
\end{equation}
If it is indeed the case that $c_{\odd} > c'_{\odd}$, we can prove it by finding suitable rational approximations of $\lambda$ and $\mu$ via inequalities \eqref{inq-coddlb} and \eqref{inq-c'oddub}.    

In the specific case $(x,y,z)=(7,1,9)$, we find
\[
\hom(P_5,T(7,1,9))=9366,~~\hom(P_7,T(7,1,9))=106302
\]
and
\[
\hom(E_5,T(7,1,9))=9492,~~\hom(E_7,T(7,1,9))=106932,
\]
and that appropriate choices for the rational approximations to $\mu$ and $\lambda$ are  
\[
\mu_\ell = 2.3363, ~\mu_u = 2.3364,~\lambda_\ell = 3.3972 \quad\text{and}\quad \lambda_u = 3.3973.
\]
Inserting these values into inequalities \eqref{inq-coddlb} and \eqref{inq-c'oddub} we obtain,
\[
c_{\odd} > \frac{229896697436000}{86112348317} > 2669 > 2668 > \frac{287092907950625}{107616977208} > c'_{\odd}.
\]
It follows that for all sufficiently large odd $n$, $P_n$ admits more $T(7,1,9)$-colorings than $E_n$ does, establishing the first claim of the theorem.

As mentioned in the introduction, we have been unable to extend this analysis to even $n$. We can use the constants in equations \eqref{eq:path_expanded} and \eqref{eq:en_expanded} to define $c_{\even} = c_1 + c_2$ as well as $d_{\even}, c'_{\even}$, and $d'_{\even}$, giving us the equivalent of equation \eqref{eq:compact_exponents} for even $n$:
\[ \hom(P_{n+4},T(x,y,z)) = c_{\even} \lambda^n + d_{\even} \mu^n \quad \text{and} \quad \hom(E_{n+4},T(x,y,z)) = c'_{\even} \lambda^n + d'_{\even} \mu^n.
\]
At $(x,y,z)=(7,1,9)$, we find $c_{\even} < c'_{\even}$, meaning that for all sufficiently large even $n$ we have 
\[
\hom(P_n, T(7,1,9)) < \hom(E_n, T(7,1,9)).
\]
Extensive computer search has not produced values $(x,y,z)$ for which $c_{\even} > c'_{\even}$.

To tackle the second part of the theorem, set $\change{H} =\widehat{T}(x,y,z)$, the graph obtained from $T(x,y,z)$ by adding a loop at the root. We have $\mathbf{a}(\change{H}) = \begin{bmatrix}1 & x & xy & xyz \end{bmatrix}$, the same as $\mathbf{a}(H)$, but now the automorphic similarity matrix becomes 
\[ 
\change{M} = \begin{bmatrix}
    1 & x & 0 & 0 \\
    1 & 0 & y & 0 \\
    0 & 1 & 0 & z \\
    0 & 0 & 1 & 0
    \end{bmatrix},
\]
with the $1$ in the $(1,1)$-entry stemming from the loop at the root. For our new choice of target graph, we also have changes to $\mathbf{h}(G_0,u_0)$ and $\mathbf{h}(G_0,v_0)$, with specifically
\[
\mathbf{h}(G_0,u_0) = 
\begin{bmatrix}
(1 + x) + x(1 + y) + x((1 + x) + y(1 + z))\\
(1 + x) + x(1 + y) + y(1 + y + z)\\
(1 + x) + y(1 + z) + z(1 + z)\\
1 + y + z
\end{bmatrix}
\]
and
\[\mathbf{h}(G_0,v_0) = 
\begin{bmatrix}
((1 + x) + x(1 + y))(1 + x)\\
((1 + x) + y(1 + z))(1 + y)\\
(1 + y + z)(1 + z)\\
1 + z
\end{bmatrix}.
\]

By \cref{cor:coefficients}, we can write
\[ \hom(P_{n+4},H) = c_1 \lambda_1^n + c_2 \lambda_2^n +c_3 \lambda_3^n +c_4 \lambda_4^n \]
and
\[ \hom(E_{n+4},H) = c'_1 \lambda_1^n + c'_2 \lambda_2^n +c'_3 \lambda_3^n +c'_4 \lambda_4^n, \]
where $\lambda_1, \lambda_2, \lambda_3$, and $\lambda_4$ are the eigenvalues of $\change{M}$. Similar to our earlier argument in this proof, if, for some choice of $x,y$, and $z$, we see (without loss of generality)
\[ \lambda_1 > \max\{|\lambda_2|, |\lambda_3|, |\lambda_4|\},\]
then for sufficiently large $n$, if $c_1>c'_1$ then $\hom(P_{n+4},\change{H}) > \hom(E_{n+4},\change{H})$.

The characteristic polynomial of $\change{M}$ is 
\[
\change{f}(t):=t^4-t^3-(x+y+z)t^2+(y+z)t+xz.
\]
The roots of $\change{f}(t)$ can be written down explicitly, albeit with expressions that are significantly more challenging than those that appear in the roots of the characteristic polynomial of $M$. Using \mathematica, we see that at $(x,y,z) = (400, 3, 800)$, we have $\lambda_1 > 28.393$ and $\max\{|\lambda_2|, |\lambda_3|, |\lambda_4|\} < 28.39$, and further that
\begin{equation} \label{inq-hat-coefficients}
c'_1 < 11653988859 < 11654710565 < c_1, 
\end{equation}
which numerically establishes the second claim of the theorem. For a more formal treatment, we use the same approach as in the proof of the first claim of the theorem: use rational approximations to the roots of the characteristic equation to rigorously verify inequality \eqref{inq-hat-coefficients}. We assume throughout the ensuing discussion that $(x,y,z)$ is such that $\change{f}(t)$ has four real roots with distinct magnitudes, and furthermore that the root of largest magnitude is positive. These conditions will be satisfied when we apply our conclusions to the specific choice $(x,y,z)=(400,3,800)$. We denote the roots by $\lambda_i$, $i=1,2,3,4$ with $\lambda_1 > \max\{|\lambda_2|, |\lambda_3|, |\lambda_4|\}$.  

For notational convenience, we set $p_n = \hom(P_{n+4},\change{T}(x,y,z))$ and $e_n = \hom(E_{n+4},\change{T}(x,y,z))$. Using \cref{thm:extension_framework} and the derivations of earlier in this proof, we have the following equalities that allow us to compute $p_n$ and $e_n$:
\begin{equation} \label{eq-pn-comp}
p_n = 
\begin{bmatrix} 1&x&xy&xyz \end{bmatrix}
\begin{bmatrix} 1&x&0&0\\1&0&y&0\\0&1&0&z\\0&0&1&0\end{bmatrix}^n 
\begin{bmatrix}
(1 + x) + x(1 + y) + x((1 + x) + y(1 + z))\\
(1 + x) + x(1 + y) + y(1 + y + z)\\
(1 + x) + y(1 + z) + z(1 + z)\\
1 + y + z
\end{bmatrix}
\end{equation}
and
\begin{equation} \label{eq-en-comp}
e_n = 
\begin{bmatrix} 1&x&xy&xyz \end{bmatrix}
\begin{bmatrix} 1&x&0&0\\1&0&y&0\\0&1&0&z\\0&0&1&0\end{bmatrix}^n 
\begin{bmatrix}
((1 + x) + x(1 + y))(1 + x)\\
((1 + x) + y(1 + z))(1 + y)\\
(1 + y + z)(1 + z)\\
1 + z
\end{bmatrix}.
\end{equation}
Recall from \cref{cor:coefficients} that there are constants $c_i$, $c'_i$, $i=1,2,3,4$, such that $p_n =\sum_{i=1}^4 c_i\lambda_i^n$ and $e_n =\sum_{i=1}^4 c'_i\lambda_i^n$. This leads us to the following system of equations that we can solve for the $c_i$'s and $c'_i$'s:
\begin{equation} \label{eq-pn-formula}
\begin{bmatrix}
1&1&1&1\\
\lambda_1&\lambda_2&\lambda_3&\lambda_4\\
\lambda_1^2&\lambda_2^2&\lambda_3^2&\lambda_4^2\\
\lambda_1^3&\lambda_2^3&\lambda_3^3&\lambda_4^3
\end{bmatrix}
\begin{bmatrix}
c_1\\c_2\\c_3\\c_4
\end{bmatrix}
=
\begin{bmatrix}
p_0\\p_1\\p_2\\p_3
\end{bmatrix}
\end{equation}
and
\begin{equation} \label{eq-en-formula}
\begin{bmatrix}
1&1&1&1\\
\lambda_1&\lambda_2&\lambda_3&\lambda_4\\
\lambda_1^2&\lambda_2^2&\lambda_3^2&\lambda_4^2\\
\lambda_1^3&\lambda_2^3&\lambda_3^3&\lambda_4^3
\end{bmatrix}
\begin{bmatrix}
c'_1\\c'_2\\c'_3\\c'_4
\end{bmatrix}
=
\begin{bmatrix}
e_1\\e_2\\e_3\\e_4
\end{bmatrix}.
\end{equation} 
Inverting the Vandermonde matrix that appears in equations \eqref{eq-pn-formula} and \eqref{eq-en-formula}, we find that its first row is
\[
\frac{1}{(\lambda_1-\lambda_2)(\lambda_1-\lambda_3)(\lambda_1-\lambda_4)}\begin{bmatrix}
-\lambda_2\lambda_3\lambda_4 & \lambda_2\lambda_3 + \lambda_2\lambda_4 + \lambda_3\lambda_4 & -(\lambda_2+\lambda_3+\lambda_4)& 1
\end{bmatrix}.
\]
This gives the following expression for $c_1-c'_1$:
\begin{equation} \label{eq-c1-minus-c'1}
\frac{-(p_0-e_0)\lambda_2\lambda_3\lambda_4 + (p_1-e_1)(\lambda_2\lambda_3 + \lambda_2\lambda_4 + \lambda_3\lambda_4) -(p_2-e_2)(\lambda_2+\lambda_3+\lambda_4) + (p_3-e_3)}{(\lambda_1-\lambda_2)(\lambda_1-\lambda_3)(\lambda_1-\lambda_4)}.
\end{equation}
Using elementary symmetric functions to express the coefficients of a polynomial in terms of its roots, we have  
\begin{eqnarray*}
\lambda_1\lambda_2\lambda_3\lambda_4 & = & xz \\
\lambda_1\lambda_2\lambda_3 + \lambda_1\lambda_2\lambda_4 + \lambda_1\lambda_3\lambda_4 +\lambda_2\lambda_3\lambda_4 & = & -y-z \\
\lambda_1\lambda_2 + \lambda_1\lambda_3 + \lambda_1\lambda_4 + \lambda_2\lambda_3 + \lambda_2\lambda_4 + \lambda_3\lambda_4 & = & -x-y-z \\
\lambda_1 + \lambda_2 + \lambda_3 + \lambda_4 & = & 1 
\end{eqnarray*}
and so
\begin{eqnarray*}
\lambda_2\lambda_3\lambda_4 & = & \frac{xz}{\lambda_1} \\
\lambda_2\lambda_3 + \lambda_2\lambda_4 + \lambda_3\lambda_4 & = & \frac{-y-z}{\lambda_1} - \frac{xz}{\lambda_1^2} \\
\lambda_2+\lambda_3+\lambda_4 & = & 1 - \lambda_1.
\end{eqnarray*}
Plugging these values into expression \eqref{eq-c1-minus-c'1}, we get
\begin{equation} \label{eq-c1-minus-c'1-2}
c_1-c'_1 = \frac{-\frac{(p_0-e_0)xz}{\lambda_1} -(p_1-e_1)\left(\frac{y+z}{\lambda_1} + \frac{xz}{\lambda_1^2}\right) -(p_2-e_2)(1 - \lambda_1) + (p_3-e_3)}{(\lambda_1-\lambda_2)(\lambda_1-\lambda_3)(\lambda_1-\lambda_4)}.
\end{equation}
Now recall that we are assuming that $\change{f}(t)$ has a unique root of largest magnitude, $\lambda_1$, that is positive. It follows that $(\lambda_1-\lambda_2)(\lambda_1-\lambda_3)(\lambda_1-\lambda_4) > 0$, and so we conclude from equation \eqref{eq-c1-minus-c'1-2} that in order to show $c_1>c'_1$ (for a particular choice of $(x,y,z)$), it suffices to show 
\[
-\frac{(p_0-e_0)xz}{\lambda_1} -(p_1-e_1)\left(\frac{y+z}{\lambda_1} + \frac{xz}{\lambda_1^2}\right) -(p_2-e_2)(1 - \lambda_1) + (p_3-e_3) > 0
\]
or equivalently
\begin{equation} \label{eq-c1-minus-c'1-3}
(p_2-e_2)\lambda_1^3 + ((p_3-e_3)-(p_2-e_2))\lambda_1^2 -((p_0-e_0)xz+(p_1-e_1)(y+z))\lambda_1 - (p_1-e_1)xz > 0. 
\end{equation}

We now specialize to $(x,y,z)=(400,3,800)$. For this choice we have $\change{f}(t)=t^4-t^3-1203t^2+803t+320000$. We can check (within the rationals) that
\[
\begin{array}{rcccl}
\change{f}(28.394) & > & 0 & > & \change{f}(28.393), \\
\change{f}(20.429) & < & 0 & < & \change{f}(20.428), \\
\change{f}(-19.435) & > & 0 & > & \change{f}(-19.436),
\end{array}
\]
and
\[
\begin{array}{rcccl}
\change{f}(-28.385) & < & 0 & < & -28.386.
\end{array}
\]
It follows that $\change{f}(t)$ has four distinct real roots, the largest of which (in magnitude), $\lambda_1$, satisfies $28.393 < \lambda_1 < 28.394$, and so for this choice of $(x,y,z)$, all the assumptions we have made thus far on the roots of the characteristic polynomial of $\change{M}$ are satisfied. We can also directly compute, via equations \eqref{eq-pn-comp} and \eqref{eq-en-comp}, 
\[
\begin{array}{rcl}
p_0-e_0 & = & 0,\\
p_1-e_1 & = & -263683600,\\
p_2-e_2 & = & -5066563600,\\
p_3-e_3 & = & 42277585600.
\end{array}
\]
It follows from equation \eqref{eq-c1-minus-c'1-3} that to show $c_1>c'_1$ in the case $(x,y,z)=(400,3,800)$, it suffices to show that $h(\lambda_1)>0$ where 
\[
h(t) = -5066563600t^3 + 47344149200t^2 + 211737930800t + 84378752000000.
\]
The discriminant of $h(t)$, which can be calculated exactly in rationals, is negative, and so $h(t)$ has only one real root. We can (in rationals) compute that $h(28.394)>0$. Given the sign of the leading coefficient of $h(t)$, it follows that $h(t)>0$ for all $t\le 28.394$, and so in particular, since $\lambda_1 < 28.394$, it follows that $h(\lambda_1)>0$. We conclude that for $(x,y,x)=(400,3,800)$ we have $c_1 > c'_1$ and so 
\[
\hom(P_{n+4},\change{T}(400,3,800)) > \hom(E_{n+4},\change{T}(400,3,800))
\] 
holds for all sufficiently large $n$. This completes the rigorous justification of the second part of \cref{thm:path_not_eventually_minimizer}.
\end{proof}

\section{Open problems} \label{sec:open_problems}

We conclude with a few questions that remain open and which we find of interest. Csikv\'{a}ri and Lin raise the following natural problem~\cite[Problem 6.2]{CsikvariLin2014}:
\begin{problem}[Csikv\'{a}ri, Lin] \label{prob:CL}
Characterize those $H$ for which $\hom(P_n,H) \le \hom(T_n,H)$ holds for all $n$ and all $T_n \in {\mathcal T}_n$.
\end{problem}
\noindent They also propose a weaker variant of~\cref{prob:CL} in which ``all $n$'' is replaced by ``all sufficiently large $n$''. In the present paper we have shown that there are $H$'s which do not satisfy the condition imposed by the weaker variant of \cref{prob:CL}. Settling the problem in general remains open.    

It is possible to use \cref{thm:extension_framework} to confirm that $E_7$ is the tree on seven vertices which admits the fewest $T(18,3,32)$-colorings. We proved in \cref{thm:path_not_eventually_minimizer} that $E_n$ admits fewer $\widehat{T}(400,3,800)$-colorings than $P_n$ when $n$ is sufficiently large, but it remains open which tree on $n$ vertices admits the fewest $\widehat{T}(400,3,800)$-colorings. 

More generally, we note that both $E_n$ and $P_n$ have large diameters. If $T_n$ has diameter 2, then $T_n$ must be a star graph and so $\hom(T_n,H) \ge \hom(P_n,H)$ by \cref{thm:siderenko}. This leads to the following question:

\begin{problem}
    For $H=\widehat{T}(400,3,800)$, or any other graph, what is the largest function $d_H(n)$ for which all trees of diameter at most $d_H(n)$ admit at least as many $H$-colorings as $P_n$?
\end{problem}

In \cref{thm:path_not_eventually_minimizer} we exhibited a target graph $H$ on $78$ vertices that is Leontovich at some $n$ (specifically $n=7$). By contrast, in~\cite{GalvinMarmorinoMcMillonNirRedlich2025} the authors prove that if $H$ has fewer than four vertices, it is not Leontovich for any $n$. This suggests the question:

\begin{problem}\label{prob:xyz}
    What is the smallest value of $m$ such that there exists a target graph $H$ on $m$ vertices that is Leontovich at some $n$?
\end{problem}

At present, we know $4 \le m \le 78$, but narrowing this range is of interest. The range of possible values of the order of the smallest strongly Leontovich graph is even larger: combining \cref{thm:path_not_eventually_minimizer} with the results of~\cite{GalvinMarmorinoMcMillonNirRedlich2025} demonstrates it must be between $4$ and $|\widehat{T}(400,3,800)| = 961601$, though no attempt was made to minimize the order of the strongly Leontovich target graph in \cref{thm:path_not_eventually_minimizer}. We note, for example, that $\widehat{T}(313,3,646)$ (which has 607847 vertices) numerically appears to be strongly Leontovich.

Finally, recall from the introduction that Csikv\'ari and Lin~\cite{CsikvariLin2014} showed that for any connected $H$ and any sequence $(T_n)_{n \in {\mathbb N}}$ of trees $T_n$ with $T_n \in {\mathcal T}_n$ we have 
\[
\liminf_{n \rightarrow \infty} \left(\frac{\hom(T_n,H)}{\hom(P_n,H)}\right)^{1/n} \geq 1.
\]
This does not rule out the possibility that $\hom(T_n,H) < \hom(P_n,H)$ infinitely often, and indeed in \cref{thm:path_not_eventually_minimizer} we find a connected target graph $H$ and a sequence of trees $(T_n)_{n \geq 4}$ with $T_n \in {\mathcal T}(n)$ with
\[
\lim_{n \to \infty} \frac{\hom(T_n,H)}{\hom(P_n,H)} = \frac{c'_1}{c_1} < 0.99994. 
\]
\begin{problem}\label{prob:ratio}
    How small can $\liminf_{n \rightarrow \infty} (\hom(T_n,H)/\hom(P_n,H))$ be made? Can it be made arbitrarily small?
\end{problem}

\section*{Acknowledgments}

This work was started at the workshop ``Graph Theory: structural properties, labelings, and connections to applications'' hosted by the American Institute of Mathematics (AIM), Pasedena CA, July 22--July 26, 2024. The authors thank AIM and the organizers of the workshop for facilitating their collaboration, and thank Aysel Erey for helpful conversations.

\end{document}